\documentclass[reqno]{amsart}
\usepackage{amsmath,amssymb}
\usepackage{color}
\usepackage{enumitem}
\usepackage[nospace]{cite}

\newcommand{\real}{\mathbb{R}}
\newcommand{\nat}{\mathbb{N}}

\newcommand{\rn}{\real^N}
\newcommand{\intrn}{\int_{\real^N}}

\newcommand{\N}{\mathbb{N}}
\newcommand{\R}{\mathbb{R}}

\newcommand{\ffi}{\varphi}

\newcommand{\D}{\Delta}
\newcommand{\lam}{\lambda}

\newcommand{\wto}{\rightharpoonup}

\newcommand{\wt}{\widetilde}
\newcommand{\diff}{\,\mathrm{d}}

\newcommand{\x}{\times}

\newcommand{\nld}[1]{\|#1\|_{L^2}}
\newcommand{\nhu}[1]{\|#1\|_{H^1}}

\let\le=\leqslant
\let\ge= \geqslant

\newcommand{\disp}{\displaystyle}
\newcommand\txt{\textstyle}

\DeclareMathOperator \im{Im}
\DeclareMathOperator \re{Re}

\DeclareMathOperator \supp{supp}

\theoremstyle{plain}
\newtheorem{theorem}{Theorem}
\newtheorem{lemma}[theorem]{Lemma}
\newtheorem{proposition}[theorem]{Proposition}

\theoremstyle{definition}
\newtheorem{remark}[theorem]{Remark}
\newtheorem*{notation}{Notation}

\numberwithin{equation}{section}

\begin{document}

\title[Blow-up solutions of an inhomogeneous NLS]
{Classification of minimal mass blow-up solutions
for an $L^2$ critical inhomogeneous NLS}

\author[V. Combet]{Vianney Combet}
\address{U.F.R. de Math\'ematiques, Universit\'e Lille 1, 59655 Villeneuve d'Ascq, France}
\email{vianney.combet@math.univ-lille1.fr}

\author[F. Genoud]{Fran\c cois Genoud}
\address{Delft Institute of Applied Mathematics \\
Delft University of Technology\\
Mekelweg~4\\
2628CD Delft, The Netherlands}
\email{S.F.Genoud@tudelft.nl}



\thanks{This work is supported by the Labex CEMPI (ANR-11-LABX-0007-01).
FG is grateful to the Labex team, and in particular to its director, Prof.~Stephan De Bi\`evre,
for their warm hospitality at Universit\'e Lille 1, where the present research was initiated.}

\begin{abstract}
We establish the classification of minimal mass blow-up solutions of the $L^2$
critical inhomogeneous nonlinear Schr\"odinger equation
\[
i\partial_t u + \Delta u + |x|^{-b}|u|^{\frac{4-2b}{N}}u = 0,
\]
thereby extending the celebrated result of Merle~\cite{m93} from the classic case $b=0$
to the case $0<b<\min\{2,N\}$, in any dimension $N\ge1$.
\end{abstract}

\maketitle


\section{Introduction}

In this paper we establish the classification of  minimal mass blow-up solutions of the
inhomogeneous nonlinear Schr\"odinger equation
\begin{equation}\label{inls}
i\partial_t u + \Delta u + |x|^{-b}|u|^{p-1}u = 0, \quad
u(0,\cdot)=u_0\in H^1(\rn),
\end{equation}
in the case $p=1+\frac{4-2b}{N}$, with $0<b<\min\{2,N\}$ and any $N\ge1$,
where the equation is $L^2$ critical, as pointed out in~\cite{g12}.
The case $b=0$ is the classic focusing NLS equation with $L^2$ critical
nonlinearity. The physical relevance of~\eqref{inls} with $b>0$ may not appear obvious
due to the singularity at $x=0$. However, this model problem plays an important role
as a limiting equation in the analysis of more general inhomogeneous problems
of the form
\[
 i\partial_t u + \Delta u + V(x)|u|^{p-1}u = 0
\]
with $V(x)\sim|x|^{-b}$ as $|x|\to\infty$, which are ubiquitous in
nonlinear optics --- see \cite{g08,g10,gs} for more details.

We consider here strong solutions $u=u(t,x)\in C^0_t H^1_x ([0,T)\x\rn)$,
where $T>0$ is the maximum time of existence of $u$.
We may simply denote by $u(t)\in H^1(\rn)$ the function $x\mapsto u(t,x)$.
The solution is called global if $T=+\infty$.
If it is not the case, the blow-up alternative
states that $\Vert u(t)\Vert_{H^1}\to\infty$ as $t\uparrow T$.
Moreover, along the flow of~\eqref{inls},
we have conservation of the $L^2$ norm, also known as the \emph{mass}:
\[
\Vert u(t)\Vert_{L^2_x}=\Vert u_0\Vert_{L^2_x},
\]
and of the \emph{energy}:
\begin{equation}\label{energy}
E(u(t))=\frac12 \intrn|\nabla u(t)|^2\diff x - \frac{1}{p+1}\intrn|x|^{-b}|u(t)|^{p+1}\diff x = E(u_0).
\end{equation}
We refer to the discussion in~\cite{g12} regarding the well-posedness theory of~\eqref{inls}
in $H^1(\rn)$. The theory is similar to the classic case $b=0$: there is
local well-posedness --- {\em i.e.} existence and uniqueness of solutions for small positive times ---
(and global for small initial data) in $H^1(\rn)$ if $1<p<1+\frac{4-2b}{N-2}$
($1<p<\infty$ if $N=1,2$); there is global well-posedness for any initial data in $H^1(\rn)$,	
provided $1<p<1+\frac{4-2b}{N}$. We are here interested in the critical case
$p=1+\frac{4-2b}{N}$.

The above invariants are related to the symmetries of~\eqref{inls} in $H^1(\rn)$.
More precisely, if $u(t,x)$ solves~\eqref{inls}, then so do:
\begin{enumerate}[label=(\alph*)]
\item $u_{t_0}(t,x)=u(t-t_0,x)$, for all $t_0\in\R$ (time translation invariance);
\item $u_{\gamma_0}(t,x)=e^{i\gamma_0}u(t,x)$, for all $\gamma_0\in\R$ (phase invariance);
\item $u_{\lambda_0}(t,x)=\lambda_0^{(2-b)/(p-1)}u(\lambda_0^2t,\lambda_0x)$, for all $\lambda_0>0$ (scaling invariance).
\end{enumerate}
Note that, unlike the classic case $b=0$,
\eqref{inls} with $b>0$ is not invariant
under space translations and Galilean transformations.

The symmetries (a) and (b) are obvious and give rise, via Noether's theorem,
to the invariance of the energy and the mass, respectively. However, it
is remarkable that~\eqref{inls} indeed has the scaling symmetry (c). In the case
$p=1+\frac{4-2b}{N}$ which will be our focus here, we have $(2-b)/(p-1)=N/2$,
and so
\[
\Vert u_{\lambda_0}(t)\Vert_{L^2_x}=\Vert u(t)\Vert_{L^2_x}, \quad\text{for all } \lambda_0>0.
\]
The symmetry (c) is then called the $L^2$ {\em scaling}, and~\eqref{inls} is said to be
$L^2$ {\em critical}.

An important feature of~\eqref{inls} is the existence of standing wave solutions.
Indeed, $u(t,x)=e^{it}\ffi(x)$ is a (global) solution of~\eqref{inls} if and only if
$\ffi\in H^1(\rn)$ solves the nonlinear elliptic equation
\begin{equation}\label{sinls}
\Delta\ffi - \ffi + |x|^{-b}|\ffi|^{\frac{4-2b}{N}}\ffi=0.
\end{equation}
There exists a unique positive and radial solution of~\eqref{sinls},
called the {\em ground state}, which we will denote by $\psi$ throughout the paper.
We refer the reader to~\cite{g12} for references about
the existence and uniqueness theory for~\eqref{sinls}.

It turns out that the ground state is a fundamental object to understand the dynamics
of~\eqref{inls}. Theorem~2.5 of~\cite{g12}
shows, for instance, that the solutions of~\eqref{inls} are global provided
\[
\Vert u_0\Vert_{L^2}<\Vert\psi\Vert_{L^2}.
\]
A crucial inequality for the proof of this theorem, which can be deduced
from Proposition~2.2 of~\cite{g12}, is
\begin{equation}\label{posE}
E(u)\ge \frac12\Vert\nabla u\Vert_{L^2}^2
\left(1-\left(\frac{\Vert u\Vert_{L^2}}{\Vert \psi\Vert_{L^2}}\right)^\frac{4-2b}{N}\right),
\quad \text{for all }\,u\in H^1(\rn).
\end{equation}
Indeed, since the $L^2$ norm and the energy are conserved,~\eqref{posE} immediately yields
an \emph{a priori} bound on $\nld{\nabla u(t)}$ in the case $\nld{u_0}<\nld{\psi}$, namely
\begin{equation} \label{subcritical}
\nld{\nabla u(t)}^2 \le 2E(u_0) \left(1-\left(\frac{\nld{u_0}}{\nld{\psi}}\right)^\frac{4-2b}{N}\right)^{-1},
\end{equation}
which implies global existence. Another interesting consequence of Proposition~2.2 of~\cite{g12}
is that $E(\psi)=0$.
Therefore, $\psi$ lies on the submanifold of $H^1(\rn)$ defined by the intersection of two
constraints, $\|u\|_{L^2}=\Vert\psi\Vert_{L^2}$ and $E(u)=0$. This manifold will be characterized in
Proposition~\ref{charact.prop} below.

On the other hand, it follows from Theorem~3.1 of~\cite{g12} that there exists a solution of~\eqref{inls}
with $\Vert u_0\Vert_{L^2}=\Vert\psi\Vert_{L^2}$ which blows up in finite time. That is,
$\Vert \psi\Vert_{L^2}$ is the minimal mass for blow-up solutions of~\eqref{inls},
which will henceforth be referred to as the {\em critical mass} for~\eqref{inls}.
Note that the proof of Theorem~3.1 of~\cite{g12} relies mainly on the pseudo-conformal
transformation applied to the standing wave $e^{it}\psi$,
and if we take into account the three invariances of~\eqref{inls}
described above, we obtain a 3-parameter
family $(S_{T,\lambda_0,\gamma_0})_{T\in\R,\lambda_0>0,\gamma_0\in\R}$ of critical mass solutions
of~\eqref{inls} which blow up in finite time, defined by
\begin{equation} \label{S1.def}
S_{T,\lambda_0,\gamma_0}(t,x) = e^{i\gamma_0} e^{i\frac{\lambda_0^2}{T-t}} e^{-i\frac{|x|^2}{4(T-t)}}
\left(\frac{\lambda_0}{T-t}\right)^{N/2}\psi\left( \frac{\lambda_0x}{T-t} \right).
\end{equation}
Note that these solutions present a self-similar profile, in the sense that, for all $t\in [0,T)$, there
exists $\lam(t)>0$ such that $|S_{T,\lambda_0,\gamma_0}(t,x)|=\lam(t)^{N/2}\psi(\lam(t)x)$. Hence,
up to a time-dependent $L^2$ rescaling,
$S_{T,\lambda_0,\gamma_0}$ keeps the same shape as $\psi$ while blowing up.
We refer to Section~\ref{pseudo-conf.sec} for more details and comments
about the pseudo-conformal transformation and the construction of these critical mass solutions.

We now state our main result.
\begin{theorem} \label{main.thm}
Let $u$ be a critical mass solution of~\eqref{inls} which blows up in finite time,
\emph{i.e.}~$\nld{u_0}=\nld{\psi}$ and there exists $T>0$ such that $\disp \lim_{t\uparrow T} \nld{\nabla u(t)}=+\infty$.

Then there exist $\lambda_0>0$ and $\gamma_0\in\R$ such that, for all $t\in [0,T)$,
\[
u(t)=S_{T,\lambda_0,\gamma_0}(t),
\]
where $S_{T,\lambda_0,\gamma_0}$ is defined in~\eqref{S1.def}.
\end{theorem}

It is worth remarking here that, since the space translation invariance of~\eqref{inls}
is broken for $b>0$, our conclusion is stronger than in the case $b=0$,
which is reflected in the absence of space translation and Galilean symmetries in~\eqref{S1.def}.
In addition, it transpires from our proof (see Step 2 in Section~\ref{proof.sec}) that all of the
solution mass concentrates at the origin in $\rn$ as the blow-up occurs.

Blow-up solutions of the $L^2$ critical NLS in
the classic case $b=0$ have been thoroughly investigated since the seminal works
of Weinstein~\cite{w82,w86}. In fact, Theorem~2.5 of~\cite{g12} extends a result of
Weinstein~\cite{w82} from the case $b=0$ to the case $0<b<\min\{2,N\}$,
and Theorem~\ref{main.thm} above extends the classification result of Merle~\cite{m93}
to the case $0<b<\min\{2,N\}$.
A comprehensive review of the theory of blow-up solutions for the classic focusing NLS can be found
in~\cite{r}, where a proof of Merle's result~\cite[Theorem~4.1]{r} is presented, which is based on more
recent arguments --- notably a refined Cauchy-Schwarz inequality due to Banica~\cite{b}.

Although more scarcely, critical mass blow-up solutions have also been investigated in the context
of inhomogeneous NLS equations by several authors. For instance,
\begin{equation}\label{k}
i\partial_t u +\D u +k(x)|u|^{4/N}u=0
\end{equation}
was considered by Merle~\cite{m96}, and later by Rapha\"el and Szeftel~\cite{rs} (in the case $N=2$),
where the inhomogeneity $k$ is supposed to be smooth, positive and bounded.
Merle~\cite{m96} derived conditions on $k$ for the localization of the concentration
point of critical mass blow-up solutions, and for the non-existence of
critical mass blow-up solutions. Rapha\"el and Szeftel~\cite{rs} proved the existence
and the classification of critical mass blow-up solutions for~\eqref{k}, provided $k$ attains
its maximum in $\rn$.
Banica, Carles and Duyckaerts~\cite{bcd} studied the problem
\[
 i\partial_t u +\D u-V(x)u +g(x)|u|^{4/N}u=0,
\]
where $V$ and $g$ satisfy strong smoothness assumptions.
Assuming that $g$ is sufficiently flat at the origin, they proved
the existence of critical mass blow-up solutions by adapting a fixed point argument
developed by Bourgain and Wang~\cite{bw} in the classic case of~\eqref{inls} with $b=0$.

It is worth noting here that problem~\eqref{inls} does not fall within the scope of \cite{m96,rs,bcd}
due to the singularity at $x=0$. Moreover, our approach strongly benefits from
the scaling properties of~\eqref{inls} --- notably the pseudo-conformal invariance, which is
not present in \cite{m96,rs,bcd}.

Our proof of Theorem~\ref{main.thm} follows the scheme outlined in~\cite{r}.
In Sections~\ref{charact.sec} and~\ref{compact.sec}, respectively,
we prove a variational characterization of the ground state of~\eqref{sinls}
and a compactness property of the flow in $H^1(\rn)$.
In Section~\ref{virial.sec} we extend the classic virial identities to the inhomogeneous
case, $b>0$. In Section~\ref{pseudo-conf.sec} we
show that~\eqref{inls} is invariant under the pseudo-conformal transformation.
Combining all these ingredients,
we give the proof of Theorem~\ref{main.thm} in Section~\ref{proof.sec}.

\begin{notation}
To avoid cumbersome exponents and indices, without further notice
we let $p=1+\frac{4-2b}{N}$ throughout the paper.
We also let $2^* = \frac{2N}{N-2}$ if $N\ge3$ and $2^*=\infty$ if $N=1,2$.
We will often denote the Lebesgue norms
$\Vert\cdot\Vert_{L^q}$ merely by $\Vert\cdot\Vert_q$, for $1\le q\le\infty$.
All the integrals will be understood to be over $\rn$, even when not specified.
For $x,y\in\rn$, we denote $x\cdot y$ their inner product, and $|x|=\sqrt{x\cdot x}$
the Euclidean norm of $x$.
The symbol $C$ will denote various positive constants, the exact value of
which is not essential to the analysis.
\end{notation}


\section{Variational characterization of the ground state} \label{charact.sec}

We start by proving the following key proposition, which gives a variational characterization of the
ground state of~\eqref{sinls}.

\begin{proposition}\label{charact.prop}
Let $v\in H^1(\rn)$ be such that
\begin{equation}\label{charactcond}
\Vert v\Vert_{L^2}=\Vert \psi\Vert_{L^2} \quad\text{and}\quad
E(v)=0.
\end{equation}
Then there exist $\lambda_0>0$ and $\gamma_0\in\real$ such that
$v(x)=e^{i\gamma_0}\lambda_0^{N/2}\psi(\lambda_0x)$.
\end{proposition}

\begin{proof}
It follows from Proposition~2.2 of~\cite{g12} that
the ground state $\psi$ of~\eqref{sinls} is a minimizer of the Weinstein
functional
\[
J(u)=\frac{\Vert \nabla u\Vert_2^2\Vert u\Vert_2^{p-1}}{\intrn |x|^{-b}|u|^{p+1}\diff x},
\]
and that $E(\psi)=0$. Therefore, for any $v\in H^1(\rn)$ satisfying~\eqref{charactcond} we have
$J(v)=J(\psi)$, so that $v$ is a minimizer of $J$. But then $|v|$ is also a minimizer, since
\begin{equation}\label{gradients}
 \Vert \nabla(|v|)\Vert_2\le \Vert \nabla v\Vert_2.
\end{equation}
Furthermore, any positive minimizer is radial thanks to a result of Hajaiej~\cite{haj}.
Indeed, suppose $v_0$ is a positive minimizer that is not radial,
and consider its Schwarz symmetrization $v_0^*$. Then Theorem~6.1 of~\cite{haj} implies that
\[
\intrn |x|^{-b}|v_0^*|^{p+1}\diff x>\intrn |x|^{-b}|v_0|^{p+1}\diff x.
\]
Since, on the other hand,
\[
\Vert \nabla v_0^*\Vert_2\le \Vert \nabla v_0\Vert_2
\quad\text{and}\quad
\Vert v_0^*\Vert_2\le\Vert v_0\Vert_2
\]
by standard properties of the Schwarz symmetrization, we get $J(v_0^*)<J(v_0)$,
a contradiction. We deduce that $|v|$ is radial. Furthermore,
the Euler-Lagrange equation expressing the fact that $|v|$ is a minimizer reads
\[
\Delta(|v|)
-\txt\left(\frac{p-1}{2}\right)\frac{\Vert\nabla(|v|)\Vert_2^2}{\Vert v\Vert_2^2}|v|
+|x|^{-b}|v|^p=0.
\]
It now follows by the scaling properties of this elliptic equation, and by
the uniqueness of its positive radial solution (see the discussion in~\cite{g12}),
that
\[
|v(x)|=\lambda_0^{N/2}\psi(\lambda_0x),
\quad\text{with} \quad \lam_0=\sqrt{\frac{p-1}{2}}\frac{\Vert\nabla(|v|)\Vert_2}{\Vert v\Vert_2}.
\]
It only remains to show that $w$ defined by $w(x)=\frac{v(x)}{|v(x)|}$ is constant on $\R^N$. To do this,
first observe that differentiating $|w|^2\equiv 1$ leads to $\re(\bar w\nabla w)\equiv 0$, and so
\[
|\nabla v|^2 =|\nabla(|v|)|^2+|v|^2|\nabla w|^2 +2|v|\nabla(|v|)\cdot\re(\bar w\nabla w)
\]
then gives
\[
 \Vert \nabla v\Vert_2^2=\Vert \nabla(|v|)\Vert_2^2+\intrn|v|^2|\nabla w|^2\diff x.
\]
Now supposing $|\nabla w|\not\equiv0$ on $\rn$, we would have
strict inequality in~\eqref{gradients}, and hence $J(|v|)<J(v)$. This contradiction shows that, indeed,
$|\nabla w|\equiv0$ on $\rn$. Hence, $w$ is constant on $\rn$, and since its modulus is $1$, we deduce
that there exists $\gamma_0\in\R$ such that $w\equiv e^{i\gamma_0}$, which completes the proof.
\end{proof}


\section{Compactness}\label{compact.sec}

The main goal of this section is to prove Proposition~\ref{compact.prop} below.
To do so, we first need some inhomogeneous estimates, reproduced in the following lemma
for the reader's convenience. The proof can be found in~\cite[Appendix~A]{g10} (for $N=1$)
and~\cite[Appendix~A]{g08} (for $N\ge2$).

\begin{lemma}\label{est.lem}
Let $0<b<\min\{2,N\}$ and $1<p<1+\frac{4-2b}{N-2}$ if $N\ge3$,
$1<p<\infty$ if $N=1,2$. Then there is a constant $C=C(N,b,p)>0$ such that
\begin{multline*}
\intrn |x|^{-b}\bigl||u|^{p-1}-|v|^{p-1}\bigr||\ffi||\xi| \diff x
    \le
    C\bigl\{\left\| |u|^{p-1}-|v|^{p-1} \right\|_{L^\beta}\|\ffi\|_{L^\gamma}\|\xi\|_{L^\gamma}\\
           +\left\| |u|^{p-1}-|v|^{p-1} \right\|_{L^\sigma}\|\ffi\|_{L^{p+1}}\|\xi\|_{L^{p+1}}\bigr\}
\end{multline*}
for all $u,v,\ffi,\xi\in H^1(\rn)$, where
\[
(p-1)\beta=\gamma\in\txt(\frac{N(p+1)}{N-b},2^*)
\quad\text{and}\quad (p-1)\sigma=p+1.
\]
\end{lemma}

To prove Proposition~\ref{compact.prop}, we also need a concentration-compactness lemma.
Minor modifications to the proof of Proposition~1.7.6 in~\cite{caz}
yield the following result.

\begin{lemma}\label{conc-comp.lem}
Let $(v_n)\subset H^1(\rn)$ satisfy
\[
\lim_{n\to\infty}\Vert v_n\Vert_{L^2}=M \quad \text{and}\quad
\sup_{n\in\nat}\Vert \nabla v_n\Vert_{L^2}<\infty.
\]
Then there is a subsequence
$(v_{n_k})$ satisfying one of the three following properties:
\begin{itemize}
\item[\emph{(V)}] $\Vert v_{n_k}\Vert_{L^q} \to 0$ as $k\to\infty$, for all $q\in(2,2^*)$.
\item[\emph{(D)}]
There exist sequences $(w_k), (z_k)\subset H^1(\rn)$ such that:
\begin{enumerate}[label=\emph{(\roman*)}]
\item $\supp(w_k)\cap\supp(z_k)=\emptyset$, for all $k\in\N$,
\item $\disp \sup_{k\in\nat}\, (\Vert w_k\Vert_{H^1}+\Vert z_k\Vert_{H^1}) <\infty$,
\item $\Vert w_k\Vert_{L^2}\to\alpha M$ and $\Vert z_k\Vert_{L^2}\to(1-\alpha)M$ as $k\to\infty$, for some $\alpha\in(0,1)$,
\item $\disp \lim_{k\to\infty}\Big| \intrn |v_{n_k}|^q - \intrn |w_k|^q - \intrn |z_k|^q \Big| = 0$, for all $q\in[2,2^*)$,
\item $\disp \liminf_{k\to\infty} \Big\{ \intrn |\nabla v_{n_k}|^2 - \intrn |\nabla w_k|^2 - \intrn |\nabla z_k|^2 \Big\} \ge 0$.
\end{enumerate}
\item[\emph{(C)}]
There exist $v\in H^1(\rn)$ and a sequence $(y_k)\subset\rn$ such that
\[
v_{n_k}(\cdot-y_k) \to v \quad\text{in} \ L^q(\rn), \quad \text{for all } q\in[2,2^*).
\]
\end{itemize}
\end{lemma}

We are now ready to prove the main result of this section.

\begin{proposition}\label{compact.prop}
Consider a sequence $(v_n)\subset H^1(\rn)$ satisfying
\begin{equation}\label{norms}
\lim_{n\to\infty}\Vert v_n\Vert_{L^2}=\Vert\psi\Vert_{L^2}, \quad
\lim_{n\to\infty}\Vert \nabla v_n\Vert_{L^2}=\Vert\nabla\psi\Vert_{L^2}, \quad
\limsup_{n\to\infty} E(v_n)\le 0.
\end{equation}
Then there exist a subsequence of $(v_n)$, still denoted $(v_n)$, and $\gamma_0\in\R$ such that
\[
\lim_{n\to\infty}\Vert v_n - e^{i\gamma_0}\psi \Vert_{H^1} = 0.
\]
\end{proposition}

\begin{proof}
The behaviour of the sequence $(v_n)$ is constrained by the concentration-compactness
principle, as stated in Lemma~\ref{conc-comp.lem}. The proof will proceed in several steps:
we will first show that property (C) holds, by ruling out (V) and (D).
Then we will show that the sequence $(y_k)$ in (C) is bounded.
Using Proposition~\ref{charact.prop}, this will lead to the desired conclusion.

\medskip
\noindent {\em Step 1: Compactness.}
Applying Lemma~\ref{est.lem} with $v=0$ and $u=\ffi=\xi=v_{n_k}$,
there exists $\gamma\in(\frac{N(p+1)}{N-b},2^*)$ such that
\[
\intrn |x|^{-b}|v_{n_k}|^{p+1}\diff x
\le C\big(\Vert v_{n_k}\Vert_\gamma^{p+1}+\Vert v_{n_k}\Vert_{p+1}^{p+1}\big).
\]
Since $\gamma,p+1\in(2,2^*)$,
(V) would imply that $\intrn |x|^{-b}|v_{n_k}|^{p+1}\diff x\to0$ and so
\begin{equation}\label{vancontr}
\lim_{k\to\infty} E(v_{n_k})
=\lim_{k\to\infty}\frac12\Vert \nabla v_{n_k}\Vert_2^2
-\frac{1}{p+1}\intrn|x|^{-b}|v_{n_k}|^{p+1}\diff x = \frac12\Vert \nabla \psi \Vert_2^2>0,
\end{equation}
which contradicts~\eqref{norms}. Therefore, (V) cannot occur.

Now suppose by contradiction that (D) holds. We claim that
\begin{equation}\label{potential}
\lim_{k\to\infty}\Big| \intrn|x|^{-b} |v_{n_k}|^{p+1}\diff x - \intrn |x|^{-b}|w_k|^{p+1} \diff x
- \intrn |x|^{-b}|z_k|^{p+1} \diff x\Big| = 0.
\end{equation}
It then follows from property (D)(v) in Lemma~\ref{conc-comp.lem} that
\begin{multline}\label{limsup}
\limsup_{k\to\infty} E(w_k)+E(z_k) \le
\frac12\liminf_{k\to\infty} \intrn |\nabla v_{n_k}|^2 \diff x\\
-\frac{1}{p+1}\liminf_{k\to\infty}\intrn |x|^{-b}|v_{n_k}|^{p+1}\diff x
\le\limsup_{k\to\infty} E(v_{n_k})\le 0.
\end{multline}
On the other hand, property (D)(iii) of Lemma~\ref{conc-comp.lem} with
$M=\Vert\psi\Vert_2$, together with inequality~\eqref{posE}, imply that
$E(w_k), E(z_k)\ge0$ for $k$ large enough, and so by~\eqref{limsup}
\[
E(w_k)\to 0 \quad\text{and}\quad E(z_k)\to 0 \quad\text{as} \ k\to\infty.
\]
But then, using again property (D)(iii) of Lemma~\ref{conc-comp.lem} and
inequality~\eqref{posE}, we see that
\[
\Vert \nabla w_k\Vert_2 \to 0 \quad\text{and}\quad \Vert \nabla z_k\Vert_2 \to 0
\quad\text{as} \ k\to\infty,
\]
which in turn implies that
\[
\lim_{k\to\infty} \intrn\! |x|^{-b} |v_{n_k}|^{p+1}\diff x =
\lim_{k\to\infty} \left( \intrn\! |x|^{-b}|w_k|^{p+1}\diff x + \intrn\! |x|^{-b}|z_k|^{p+1}\diff x\right) = 0,
\]
again leading to the contradiction~\eqref{vancontr}. Thus, to rule out
(D), we need only prove claim~\eqref{potential}, which we do now.
Defining $\xi_k=v_{n_k}-w_k-z_k$, it follows from the construction of
the sequences $w_k$ and $z_k$ in the proof of~\cite[Proposition~1.7.6]{caz}
that
\[
\big| |v_{n_k}|^{p+1}-|w_k|^{p+1}-|z_k|^{p+1} \big|
\le C |v_{n_k}|^p|\xi_k|
\]
and
\[
\Vert\xi_k\Vert_2\to0 \quad\text{as} \ k\to\infty.
\]
Since $\Vert\nabla\xi_k\Vert_2$ is bounded by property (D)(v), the Gagliardo-Nirenberg
inequality then implies that
\[
\Vert\xi_k\Vert_q\to0 \quad\text{as} \ k\to\infty, \quad\forall\,q\in[2,2^*).
\]
Hence, it follows from Lemma~\ref{est.lem} that
\begin{align*}
&\Big| \intrn|x|^{-b} |v_{n_k}|^{p+1} \diff x - \intrn |x|^{-b}|w_k|^{p+1}\diff x
- \intrn |x|^{-b}|z_k|^{p+1}\diff x \Big| \\
& \le C\intrn|x|^{-b} |v_{n_k}|^p  |\xi_k| \diff x
\le C\big(\Vert v_{n_k}\Vert_\gamma^p\Vert \xi_k\Vert_\gamma
+\Vert v_{n_k}\Vert_{p+1}^p\Vert \xi_k\Vert_{p+1}\big)
\to0 \ \text{as} \ k\to\infty,
\end{align*}
which proves the claim. Therefore, we conclude from Lemma~\ref{conc-comp.lem}
that there exist $v\in H^1(\rn)$ and a sequence $(y_k)\subset\real^N$ such that
\begin{equation}\label{compact}
v_{n_k}(\cdot-y_k) \to v \quad\text{in} \ L^q(\rn), \quad \forall\,q\in[2,2^*).
\end{equation}

\medskip
\noindent
{\em Step 2: Localization.}
We will now show that $(y_k)\subset\real^N$ is bounded.
Suppose by contradiction that $|y_k|\to\infty$ as $k\to\infty$
(up to a subsequence).
Note that $E(v_{n_k})$ can be written as
\begin{equation}\label{translate}
E(v_{n_k})=\frac12\Vert\nabla v_{n_k}\Vert_2^2
-\frac{1}{p+1}\intrn|x-y_k|^{-b}|\wt{v}_{n_k}(x)|^{p+1}\diff x,
\end{equation}
with $\wt{v}_{n_k}(x)=v_{n_k}(x-y_k)$. We will show that the second term
in the right-hand side of~\eqref{translate} goes to zero as $k\to\infty$, so that
\begin{equation}\label{contrE}
E(v_{n_k})\to \frac12\Vert\nabla \psi\Vert_2^2>0 \quad\text{as} \ k\to\infty,
\end{equation}
which contradicts~\eqref{norms}. We split the integral into two parts, as
\[
\underbrace{\int_{|x-y_k|< R}|x-y_k|^{-b}|\wt{v}_{n_k}(x)|^{p+1}\diff x}_{\mathrm{I}} +
\underbrace{\int_{|x-y_k|\ge R}|x-y_k|^{-b}|\wt{v}_{n_k}(x)|^{p+1}\diff x}_{\mathrm{II}},
\]
for some $R>0$. First, by H\"older's inequality,
\begin{equation}\label{holdI}
\mathrm{I} \le \Big( \int_{|x-y_k|< R} |x-y_k|^{-b\alpha}\diff x\Big)^\frac{1}{\alpha}
\Big( \int_{|x-y_k|< R} |\wt{v}_{n_k}(x)|^{(p+1)\beta}\diff x\Big)^\frac{1}{\beta}
\end{equation}
where $\alpha,\beta\ge 1$ satisfy $\frac{1}{\alpha}+\frac{1}{\beta}=1$. Now the first
factor in the right-hand side of~\eqref{holdI} is finite provided $\beta>\frac{N}{N-b}$.
In fact it is possible to choose $\beta$ so that $\beta(p+1)\in(\frac{N(p+1)}{N-b},2^*)$
and it follows from~\eqref{compact} that
\[
\int_{|x-y_k|< R} |\wt{v}_{n_k}(x)|^{(p+1)\beta}\diff x \to 0 \quad\text{as} \ k\to\infty.
\]
On the other hand,
\[
\mathrm{II} \le R^{-b} \intrn |\wt{v}_{n_k}(x)|^{p+1} \diff x \le C R^{-b}
\]
by the Sobolev embedding theorem and the boundedness of $(\wt{v}_{n_k})$
in $H^1(\rn)$. Hence, $\mathrm{II}$ can be made arbitrarily small by choosing $R$ large
enough, uniformly in $k$. This completes the proof of~\eqref{contrE}, and we conclude
that the sequence $(y_k)$ is bounded in $\real^N$.

\medskip
\noindent
{\em Step 3: Conclusion.}
By passing to a subsequence if necessary, we can suppose that $y_k\to y^*$ as $k\to\infty$,
for some $y^*\in\rn$. Hence,
\[
v_{n_k}\to v^*=v(\cdot+y^*) \quad\text{in} \ L^q(\rn), \quad\forall\, q\in[2,2^*).
\]
Furthermore, we can also suppose that $v_{n_k} \wto v^*$ weakly in
$H^1(\rn)$, and it follows from~\eqref{norms} that
\[
\Vert v^*\Vert_2 = \Vert \psi\Vert_2 \quad\text{and}\quad
\Vert\nabla v^*\Vert_2\le \Vert \nabla\psi\Vert_2.
\]
Now, by Lemma~2.1 of~\cite{g12},
$v\mapsto \intrn |x|^{-b} |v|^{p+1}\diff x$
is weakly sequentially continuous and we have, by~\eqref{norms},
\begin{align*}
E(v^*)&=\frac12\Vert\nabla v^*\Vert_2^2-\frac{1}{p+1}\intrn |x|^{-b} |v^*|^{p+1}\diff x\\
&\le \frac12\Vert\nabla\psi\Vert_2^2-\frac{1}{p+1}\intrn |x|^{-b} |v^*|^{p+1}\diff x
=\lim_{k\to\infty} E(v_{n_k})\le 0.
\end{align*}
But it follows from~\eqref{posE} that $E(v^*)\ge0$, and so
\[
E(v^*) = 0\quad \text{and}\quad \|\nabla v^*\|_2=\|\nabla\psi\|_2.
\]
Together with $\|v^*\|_2=\|\psi\|_2$,
these two last identities imply, by Proposition~\ref{charact.prop}, that
$v^*=e^{i\gamma_0}\psi$ for some $\gamma_0\in\real$.
Finally, since we now have $\nhu{v_{n_k}}\to \nhu{\psi}=\nhu{v^*}$, $(v_{n_k})$ converges strongly
to $v^*$ in $H^1(\rn)$, which concludes the proof of Proposition~\ref{compact.prop}.
\end{proof}

\begin{remark}
It is worth noting here that Proposition~\ref{compact.prop} can be used to prove
the `orbital stability' of the ground state $\psi$, in the sense of Theorem~3.7 of~\cite{r}.
Moreover, the notion of stability is stronger here, since we can take
$x(t)\equiv0$ for the translation shift appearing in Theorem~3.7 of~\cite{r}.
\end{remark}


\section{Virial identities} \label{virial.sec}

Let us define
\[
\Sigma = \{ u\in H^1(\rn)\ |\ xu\in L^2(\rn) \}.
\]
Then, for $u(t)\in\Sigma$, the quantity
\begin{equation} \label{Gamma.def}
\Gamma(t) = \intrn |x|^2|u(t,x)|^2\diff x
\end{equation}
is well defined, and when $u$ is a solution of the classic, homogeneous, NLS equation
(\emph{i.e.}~\eqref{inls} with $b=0$),
it is well known that $\Gamma'$ and $\Gamma''$ have simple expressions,
very useful to prove blow-up results when $u_0\in\Sigma$.
The following lemma shows that $\Gamma$ is still a key quantity in
the inhomogeneous case, $b>0$.

\begin{lemma} \label{virial.lem}
Let $u$ be a solution of~\eqref{inls} defined on $[0,T)$, such that $u(t)\in\Sigma$ for every $t\in [0,T)$.
Then, for all $t\in [0,T)$, we have
\begin{equation} \label{virial1}
\Gamma'(t) = 4\im \intrn \bar u(t,x)(\nabla u(t,x)\cdot x)\diff x
\end{equation}
and
\begin{equation} \label{virial2}
\Gamma''(t) = 16E(u(t)) + \frac{4}{p+1}(N-Np-2b+4)\intrn |x|^{-b}|u(t,x)|^{p+1}\diff x.
\end{equation}
\end{lemma}

\begin{proof}
By regularization, we may assume $u$ smooth for the following calculation.
Since $u$ satisfies~\eqref{inls}, we first find
\begin{align*}
\Gamma'(t) &= 2\re\int |x|^2 \bar u\partial_t u = 2\re\int |x|^2\bar u i (\D u+|x|^{-b}|u|^{p-1}u) \\
&= -2\im\int |x|^2(\bar u\D u +|x|^{-b}|u|^{p+1}) = -2\im\int |x|^2 \bar u\D u.
\end{align*}
Integrating by parts and using $\nabla |x|^2 = 2x$, we obtain
\[
\Gamma'(t) = 2\im\int \nabla u\cdot (\bar u \nabla |x|^2 +|x|^2\nabla\bar u)
= 4\im\int \bar u(\nabla u\cdot x).
\]
Using again an integration by parts and denoting $\nabla\cdot v = \sum_j \partial_{x_j}v_j$, we now compute
\begin{align*}
\Gamma''(t) &= 4\im\int \partial_t\bar u(x\cdot\nabla u)+\bar u(x\cdot\nabla\partial_t u)
= 4\im\int \partial_t u [-x\cdot\nabla \bar u -\nabla\cdot(\bar ux)] \\
&= 4\im\int \partial_t u[-2x\cdot\nabla\bar u -\bar u\nabla\cdot x]
= -8\im\int \partial_t u(x\cdot\nabla\bar u) -4N\im\int \partial_t u\bar u.
\end{align*}
To compute these last two terms, we use~\eqref{inls} and first find
\begin{align}
-4N\im\int \partial_t u\bar u &= -4N\re\int \bar u(\D u+|x|^{-b}|u|^{p-1}u) \notag \\
&= 4N\int |\nabla u|^2 -4N\int |x|^{-b}|u|^{p+1}. \label{virialpart1}
\end{align}
Similarly, we also find
\begin{align*}
-8\im\int \partial_t u(x\cdot\nabla\bar u) &= -8\re\int (x\cdot\nabla\bar u)(\D u+|x|^{-b}|u|^{p-1}u) \\
&= -8\re\int \D u(x\cdot\nabla\bar u) -8\int |x|^{-b}x\cdot |u|^{p-1}\re(u\nabla\bar u) \\
&= A+B.
\end{align*}
Since $\nabla(|u|^{p+1}) = (p+1)|u|^{p-1}\re(u\nabla\bar u)$ and $\nabla(|x|^{-b})=-b|x|^{-b-2}x$,
we find by an integration by parts
\begin{align}
B &= \frac{8}{p+1}\int |u|^{p+1}\nabla\cdot(|x|^{-b}x)
= \frac{8}{p+1}\int |u|^{p+1}(\nabla |x|^{-b}\cdot x+|x|^{-b}\nabla\cdot x) \notag \\
&= \frac{8}{p+1}\int |u|^{p+1}(-b|x|^{-b}+N|x|^{-b}) = \frac{8(N-b)}{p+1} \int |x|^{-b}|u|^{p+1}. \label{virialpart2}
\end{align}
Similarly, since $\partial_{x_k}(|\partial_{x_j}u|^2) = 2\re(\partial_{x_j}u\partial_{x_j}\partial_{x_k}\bar u)$
for $1\le j,k\le N$, we find
\begin{align}
A &= -8\sum_{j,k} \re\int \partial_{x_j}^2u\, x_k\partial_{x_k}\bar u
= 8\sum_{j,k} \re\int \partial_{x_j}u (\delta_{j,k}\partial_{x_k}\bar u + x_k\partial_{x_j}\partial_{x_k}\bar u) \notag \\
&= 8\sum_j \int |\partial_{x_j}u|^2 +4\sum_{j,k} \int x_k\partial_{x_k}(|\partial_{x_j}u|^2) \notag \\
&= 8\sum_j \int |\partial_{x_j}u|^2 -4\sum_{j,k} \int |\partial_{x_j}u|^2 \label{virialpart3}
= (8-4N)\int |\nabla u|^2,
\end{align}
where we wrote $\delta_{j,k}=1$ for $j=k$ and $0$ otherwise.
Finally, gathering~\eqref{virialpart1},~\eqref{virialpart2} and~\eqref{virialpart3}, we obtain
\begin{align*}
\Gamma''(t) &= 8\int |\nabla u|^2 + \frac{4}{p+1}(N-Np-2b)\int |x|^{-b}|u|^{p+1} \\
&= 16E(u) +\frac{4}{p+1}(N-Np-2b+4)\int |x|^{-b}|u|^{p+1},
\end{align*}
from the definition~\eqref{energy} of the energy, which concludes the proof of the lemma.
\end{proof}

\begin{remark}
\rm
Note that the previous lemma is valid for any energy subcritical value of $p$,
\emph{i.e.}~with the restriction $p<1+\frac{4-2b}{N-2}$ if $N\ge 3$. In our $L^2$ critical case,
where $p=1+\frac{4-2b}{N}$, identity~\eqref{virial2} simply reduces, thanks to~\eqref{energy}, to
\begin{equation} \label{GammaSeconde}
\Gamma''(t) = 16E(u_0),
\end{equation}
which is also the key identity to classify the blow-up solutions in the homogeneous case, $b=0$.
\end{remark}


\section{Pseudo-conformal transformation} \label{pseudo-conf.sec}

We now establish the pseudo-conformal invariance of~\eqref{inls},
which was observed by the second author in~\cite[Section 3]{g12}.
For the reader's convenience, and also to be coherent with the notation of the present
paper, we prove the following statement.

\begin{lemma} \label{pseudo-conf.lem}
Let $u$ be a global solution of~\eqref{inls}.
Then, for all $T\in\R$, the function~$u_T$ defined by
\[
u_T(t,x) = \frac{e^{-i\frac{|x|^2}{4(T-t)}}}{(T-t)^{N/2}} u\left(\frac{1}{T-t},\frac{x}{T-t}\right)
\]
is also a solution of~\eqref{inls}, defined on $(-\infty,T)$, and has the same mass as $u$.
\end{lemma}

\begin{proof}
A straightforward calculation gives first
\[
\partial_t u_T(t,x) = \frac{e^{-i\frac{|x|^2}{4(T-t)}}}{(T-t)^{N/2+2}} \left[ \frac{N}{2}(T-t)u -i\frac{|x|^2}{4}u
+\partial_t u +x\cdot\nabla u \right]\! \left(\frac{1}{T-t},\frac{x}{T-t}\right).
\]
We also find
\[
\D u_T(t,x) = \frac{e^{-i\frac{|x|^2}{4(T-t)}}}{(T-t)^{N/2+2}} \left[ -\frac{|x|^2}{4}u -i\frac{N}{2}(T-t)u
-ix\cdot\nabla u +\D u\right]\! \left(\frac{1}{T-t},\frac{x}{T-t}\right)
\]
and, since $\frac{N}{2}\left(\frac{4-2b}{N}\right) = 2-b$,
\begin{align*}
|x|^{-b}|u_T|^{\frac{4-2b}{N}} u_T(t,x) &= |x|^{-b}\frac{1}{(T-t)^{2-b}} \frac{e^{-i\frac{|x|^2}{4(T-t)}}}{(T-t)^{N/2}}
|u|^{\frac{4-2b}{N}} u \left(\frac{1}{T-t},\frac{x}{T-t}\right) \\
&= \left| \frac{x}{T-t} \right|^{-b} \frac{e^{-i\frac{|x|^2}{4(T-t)}}}{(T-t)^{N/2+2}}
|u|^{\frac{4-2b}{N}} u \left(\frac{1}{T-t},\frac{x}{T-t}\right).
\end{align*}
It follows that
\begin{multline*}
i\partial u_T(t,x) +\D u_T(t,x) + |x|^{-b}|u_T|^{\frac{4-2b}{N}}u_T(t,x) \\
= \frac{e^{-i\frac{|x|^2}{4(T-t)}}}{(T-t)^{N/2+2}}
\left[ i\partial_t u +\D u + |x|^{-b} |u|^{\frac{4-2b}{N}} u\right]\! \left(\frac{1}{T-t},\frac{x}{T-t}\right) = 0,
\end{multline*}
since $u$ satisfies~\eqref{inls}. The fact that $\Vert u_T(t)\Vert_{L^2_x}=\Vert u(t)\Vert_{L^2_x}$
follows from the $L^2$~scaling invariance.
\end{proof}

We can now construct, as announced in the introduction, a 3-parameter family of critical mass solutions of~\eqref{inls}
which blow up in finite time.

\begin{proposition} \label{S.prop}
For all $T\in\R$, $\lambda_0>0$ and $\gamma_0\in\R$, the function $S_{T,\lambda_0,\gamma_0}$, defined by
\begin{equation} \label{S2.def}
S_{T,\lambda_0,\gamma_0}(t,x) = e^{i\gamma_0} e^{i\frac{\lambda_0^2}{T-t}} e^{-i\frac{|x|^2}{4(T-t)}}
\left(\frac{\lambda_0}{T-t}\right)^{N/2}\psi\left( \frac{\lambda_0x}{T-t} \right),
\end{equation}
is a critical mass solution of~\eqref{inls} defined on $(-\infty,T)$, and which blows up with speed
\[
\nld{\nabla S_{T,\lambda_0,\gamma_0}(t)} \sim \frac{C}{T-t} \quad \text{as } t\uparrow T,
\]
for some $C>0$.
\end{proposition}

\begin{proof}
The proposition is a simple consequence of Lemma~\ref{pseudo-conf.lem} applied to the global solution
\[
u_{\lambda_0,\gamma_0}(t,x) = e^{i\gamma_0} e^{i\lambda_0^2 t}\lambda_0^{N/2}\psi(\lambda_0x),
\]
which is nothing more than the renormalized version of the standing wave $u(t,x)=e^{it}\psi(x)$
under the scaling and phase symmetries.
\end{proof}

\begin{remark}
Note that the blow-up solutions of the family exhibited in Proposition~\ref{S.prop} can all be retrieved from the solution
\[
S(t,x) := S_{0,1,0}(t,x) = e^{i\frac{|x|^2}{4t}} e^{-\frac{i}{t}} \frac{1}{|t|^{N/2}}\psi\left(-\frac{x}{t}\right),
\]
defined on $(-\infty,0)$ and which blows up at $t=0$ with speed
\[
\nld{\nabla S(t)} \sim \frac{C}{|t|}\quad \text{as } t\uparrow0,
\]
for some $C>0$. Indeed, all the solutions $S_{T,\lambda_0,\gamma_0}$ are equal to $S$,
up to the symmetries (a), (b) and (c) stated in the introduction.
Namely, if we apply the changes $u(t,x)\to \lambda_0^{-N/2}u(\lambda_0^{-2}t,\lambda_0^{-1}x)$, $u(t,x)\to u(t-T,x)$
and finally $u(t,x)\to e^{i\gamma_0}u(t,x)$ to $S$, we obtain $S_{T,\lambda_0,\gamma_0}$.
\end{remark}


\section{Proof of Theorem~\ref{main.thm}} \label{proof.sec}

Before starting the proof of Theorem~\ref{main.thm}, we need
to control the $L^2$ norm of the gradient of our solution by its energy in the case
$\nld{u}=\nld{\psi}$, for which~\eqref{posE} does not imply~\eqref{subcritical} anymore.
The following observation of Banica~\cite{b} is relevant in this context.
For $u\in H^1(\rn)$, $\theta\in C_0^{\infty}(\rn)$ real-valued and $s\in\R$,
we have $\nabla(ue^{is\theta}) = (\nabla u+isu\nabla\theta)e^{is\theta}$, and so
\[
|\nabla(ue^{is\theta})|^2 = |\nabla u|^2 +2s\nabla\theta\cdot\im(\bar u\nabla u) +s^2|\nabla\theta|^2|u|^2.
\]
By integrating with respect to $x\in\rn$, we get
\begin{equation} \label{Banica}
E(ue^{is\theta}) = E(u) +s\int \nabla\theta\cdot\im(\bar u\nabla u) +\frac{s^2}{2}\int |\nabla\theta|^2|u|^2.
\end{equation}
We can now easily prove the following refined Cauchy-Schwarz inequality for critical mass functions.

\begin{lemma} \label{Banica.lem}
Let $u\in H^1(\rn)$ be a function such that $\nld{u}=\nld{\psi}$. Then, for all $\theta\in C_0^{\infty}(\rn)$, one has
\[
\left| \int \nabla\theta\cdot\im(\bar u\nabla u)\right| \le \sqrt{2E(u)} \left(\int |\nabla\theta|^2|u|^2\right)^{1/2}.
\]
\end{lemma}

\begin{proof}
For all $s\in\R$, we now have $\nld{ue^{is\theta}} = \nld{u} = \nld{\psi}$, so $E(ue^{is\theta})\ge 0$ and $E(u)\ge0$ by~\eqref{posE}.
The result follows from the quadratic polynomial expression~\eqref{Banica} in $s$ of $E(ue^{is\theta})$,
which thus must have a non-positive discriminant.
\end{proof}

We now have all the tools to prove Theorem~\ref{main.thm}.
Let $u$ be a solution of~\eqref{inls} such that $\nld{u_0}=\nld{\psi}$
and which blows up in finite time: there exists $T>0$ such that
$\lim_{t\uparrow T} \nld{\nabla u(t)} = +\infty$. The core idea of the proof is to integrate
the equation backwards in time, from the blow-up time, in order to show that $u_0$ belongs to
the family of solutions~\eqref{S2.def}. We shall proceed in four steps.

\medskip
\noindent {\em Step 1: Compactness of the flow in $H^1$.} Let $(t_n)\subset\R$ be a sequence of times such that $t_n\uparrow T$ as $n\to+\infty$.
Then we set
\[
u_n=u(t_n),\quad \lambda_n = \frac{\nld{\nabla u_n}}{\nld{\nabla \psi}},\quad
v_n(x) = \lambda_n^{-N/2}u_n(\lambda_n^{-1}x).
\]
First note that $\lambda_n\to +\infty$ as $n\to+\infty$, and $\nld{v_n}=\nld{u_n}=\nld{u_0}=\nld{\psi}$ from the $L^2$ scaling.
With the change of variables $y=\lambda_n^{-1}x$, we find
\[
\nld{\nabla v_n}^2 = \lambda_n^{-2}\int |\nabla u_n(y)|^2\diff y = \lambda_n^{-2}\nld{\nabla u_n}^2 = \nld{\nabla \psi}^2
\]
and, since $p=1+\frac{4-2b}{N}$,
\begin{align*}
E(v_n) &= \frac{\lambda_n^{-2}}{2} \int |\nabla u_n(y)|^2\diff y
-\frac{1}{p+1}\left(\lambda_n^{N-b-\frac{N(p+1)}{2}}\right)\int |y|^{-b} |u_n(y)|^{p+1}\diff y \\
&= \frac{E(u_n)}{\lambda_n^2} = \frac{E(u_0)}{\lambda_n^2} \xrightarrow[n\to+\infty]{} 0.
\end{align*}
Hence, we can apply Proposition~\ref{compact.prop} to $(v_n)$, which gives $\gamma_2\in\R$ such that,
up to extracting a subsequence of $(v_n)$, we have
\begin{equation} \label{vn}
\lim_{n\to +\infty} \nhu{v_n -e^{i\gamma_2}\psi} =0.
\end{equation}

\medskip
\noindent {\em Step 2: Mass concentration.} We can now prove that $u_n$
concentrates all of its mass at $x=0$ as $n\to+\infty$.
More precisely we show, in the sense of distributions, that
\[
|u_n| \xrightarrow[n\to+\infty]{} \nld{\psi}\delta_0.
\]
Indeed, for $\ffi\in C_0^{\infty}(\rn)$, using again the change of variables $y=\lambda_n^{-1}x$, we find
\begin{align*}
\intrn |u_n(y)|^2\ffi(y)\diff y &= \intrn |v_n(x)|^2\ffi(\lambda_n^{-1}x)\diff x \\
&= \intrn (|v_n(x)|^2-|\psi(x)|^2)\ffi(\lambda_n^{-1}x)\diff x +\intrn |\psi(x)|^2\ffi(0)\diff x \\
&\qquad +\intrn |\psi(x)|^2 [\ffi(\lambda_n^{-1}x)-\ffi(0)]\diff x.
\end{align*}
Hence, we have
\begin{multline*}
\left| \intrn |u_n(y)|^2\ffi(y)\diff y - \nld{\psi}^2\ffi(0) \right| \le \|\ffi\|_{L^{\infty}}\intrn \left||v_n(x)|^2-|\psi(x)|^2\right| \diff x \\
+ \intrn |\psi(x)|^2|\ffi(\lambda_n^{-1}x)-\ffi(0)|\diff x.
\end{multline*}
We conclude this step by noticing that $|v_n|^2$ converges to $|\psi|^2$ strongly in $L^1(\rn)$ from~\eqref{vn},
so the first integral converges to $0$ as $n\to +\infty$. Since $\lambda_n\to +\infty$, the second integral also
converges to $0$ as $n\to +\infty$ by the dominated convergence theorem, and so
\begin{equation} \label{mass.conc}
\intrn |u_n(y)|^2\ffi(y)\diff y \xrightarrow[n\to+\infty]{} \nld{\psi}^2\ffi(0).
\end{equation}

\medskip
\noindent {\em Step 3: $u(t)\in\Sigma$ for all $t\in [0,T)$.} Let $\phi\in C_0^{\infty}(\rn)$ be radial and non-negative such that $\phi(x)=|x|^2$ for $|x|\le 1$.
In other words, there exists $f\in C_0^{\infty}(\R,\R_+)$ such that $\phi(x)=f(|x|)$ and $f(r)=r^2$ for $-1\le r\le 1$.

Since $f\ge 0$, there exists $C>0$ such that $|f'(r)|^2\le Cf(r)$ for all $r\in\R$, and so
\[
|\nabla\phi(x)|^2\leq C\phi(x)
\]
for all $x\in\rn$.
Indeed, by Taylor's formula, for all $r\in\R$ and $h\in\R$, there exists $y$ between $r$ and $r+h$ such that
\[
0\le f(r+h) = f(r)+f'(r)h+\frac{f''(y)}{2}h^2 \le f(r) +f'(r)h +C'h^2,
\]
where $C'=1+\max_{r\in\R} \frac{|f''(r)|}{2}>0$. Hence the right-hand side is a non-negative quadratic polynomial in $h$,
so we must have $|f'(r)|^2-4C'f(r)\le 0$, that is, $|f'(r)|^2\le Cf(r)$ with $C=4C'>0$.

Now, for $R>0$, we define $\phi_R(x)=R^2\phi\left(\frac{x}{R}\right)$ and, for all $t\in [0,T)$,
\[
\Gamma_R(t) = \intrn \phi_R(x)|u(t,x)|^2\diff x.
\]
Note that we now have $\phi_R(x)=|x|^2$ for $|x|\le R$, and still $\phi_R\in C_0^{\infty}(\rn)$ and $|\nabla \phi_R|^2 \le C\phi_R$.
Moreover, as for the proof of~\eqref{virial1}, we find
\begin{align*}
\Gamma'_R(t) &= 2\re\int \phi_R \bar u \partial_t u = -2\im\int \phi_R \bar u(\D u +|x|^{-b}|u|^{p-1}u) \\
&= 2\im\int \nabla u\cdot(\nabla\phi_R\bar u+\phi_R\nabla\bar u) = 2\int \nabla\phi_R\cdot \im(\bar u\nabla u).
\end{align*}
Since $\nld{u}=\nld{\psi}$, we may apply Lemma~\ref{Banica.lem} and we get, since $|\nabla \phi_R|^2 \le C\phi_R$,
\[
|\Gamma'_R(t)| \le 2\sqrt{2E(u)} \left(\int |\nabla\phi_R|^2|u|^2\right)^{1/2} \le C\sqrt{E(u_0)}\sqrt{\Gamma_R(t)}.
\]
Integrating between a fixed $t\in [0,T)$ and $t_n$, we obtain
\[
|\sqrt{\Gamma_R(t)}-\sqrt{\Gamma_R(t_n)}| \le C|t-t_n|.
\]
But from the mass concentration result~\eqref{mass.conc} in Step~2, we get
\[
\Gamma_R(t_n) = \intrn |u_n(x)|^2\phi_R(x)\diff x \xrightarrow[n\to+\infty]{} \nld{\psi}^2\phi_R(0)=0.
\]
Thus, letting $n\to+\infty$ in the last inequality, we obtain, for all $t\in [0,T)$ and all $R>0$,
\[
\Gamma_R(t) \leq C(T-t)^2.
\]
Since the right-hand side of the last expression is independent of $R$, we obtain, by letting $R\to+\infty$,
for all $t\in [0,T)$,
\[
u(t)\in\Sigma \quad\text{and}\quad 0\le \Gamma(t) \le C(T-t)^2,
\]
where $\Gamma$ is defined by~\eqref{Gamma.def}.
From this estimate, we can extend by continuity $\Gamma(t)$ at $t=T$ by setting $\Gamma(T)=0$,
from which we also obtain $\Gamma'(T)=0$.
Moreover, since $u(t)\in\Sigma$ and $u$ is a solution of~\eqref{inls}, we may apply Lemma~\ref{virial.lem},
and by~\eqref{GammaSeconde} we obtain $\Gamma''(t)=16E(u_0)$,
which finally gives, for all $t\in [0,T)$,
\[
\Gamma(t) = 8E(u_0)(T-t)^2.
\]
Letting $t=0$, we find, using identity~\eqref{virial1},
\[
\Gamma(0)=\int |x|^2|u_0|^2= 8E(u_0)T^2\quad \text{and}\quad \Gamma'(0)= 4\int x\cdot\im(\overline{u_0}\nabla u_0) = -16E(u_0)T.
\]

\medskip
\noindent {\em Step 4: Determination of $u_0$ and conclusion.}
We finally apply identity~\eqref{Banica} to $u_0$ and $s=\frac{1}{2T}$,
with $\theta(x)=\frac{|x|^2}{2}$. Since $\nabla\theta(x)=x$, we obtain
\begin{align*}
E(u_0e^{i\frac{|x|^2}{4T}}) &= E(u_0) +\frac{1}{2T}\int x\cdot\im(\overline{u}_0\nabla u_0) +\frac{1}{8T^2}\int |x|^2|u_0|^2 \\
&= E(u_0) +\frac{1}{2T}(-4E(u_0)T) +\frac{1}{8T^2}(8E(u_0)T^2) \\
&= E(u_0) -2E(u_0) +E(u_0) = 0.
\end{align*}
(Note that this calculation justifies, {\em a posteriori}, the application of~\eqref{Banica}
with the function $\theta(x)=\frac{|x|^2}{2}\not\in C_0^{\infty}(\rn)$.)
Hence, we have $\nld{u_0e^{i\frac{|x|^2}{4T}}}=\nld{\psi}$ and $E(u_0e^{i\frac{|x|^2}{4T}})=0$,
and we deduce
from Proposition~\ref{charact.prop} that there exist $\lambda_1>0$ and $\gamma_1\in\R$ such that
\[
u_0(x) = e^{i\gamma_1} e^{-i\frac{|x|^2}{4T}} \lambda_1^{N/2} \psi(\lambda_1x).
\]
Now the end of the proof entirely relies on the pseudo-conformal transformation,
as stated in Section~\ref{pseudo-conf.sec}. Indeed, if we define $\lambda_0=\lambda_1T>0$ and
$\gamma_0 = \gamma_1-\lambda_1^2T = \gamma_1-\frac{\lambda_0^2}{T} \in\R$, we can rewrite $u_0$ as
\[
u_0(x) = e^{i\gamma_0} e^{i\frac{\lambda_0^2}{T}} e^{-i\frac{|x|^2}{4T}} \left(\frac{\lambda_0}{T}\right)^{N/2}
\psi\left( \frac{\lambda_0x}{T}\right),
\]
so that $u_0=S_{T,\lambda_0,\gamma_0}(0)$, where $S_{T,\lambda_0,\gamma_0}$ is defined by~\eqref{S2.def}.
It then follows
by uniqueness of the solution of~\eqref{inls} that $u(t)=S_{T,\lambda_0,\gamma_0}(t)$
for all $t\in [0,T)$, which concludes the proof of Theorem~\ref{main.thm}.\hfill$\Box$


\end{document}